\documentclass[10pt,twoside,final]{amsart}

\usepackage[english]{babel}
\usepackage{graphicx,epstopdf,epsfig}
\usepackage{amsfonts,epsfig,fancyhdr,graphics,amsmath,amssymb}

\title{Products of involutons  in symplectic groups II}

\newtheorem{definition}{Definition}[section]
\newtheorem{theorem}{Theorem}[section]
\newtheorem{proposition}[theorem]{Proposition}
\newtheorem{lemma}[theorem]{Lemma}
\newtheorem{corollary}[theorem]{Corollary}
\newtheorem{remark}[theorem]{Remark}

\newcommand {\ch }{\operatorname{char}}
\newcommand {\SpG }{\operatorname{Sp}}

\newcommand {\GL }{\mathrm{GL}}

\newcommand {\Fix }{\mathrm{Fix}}

\newcommand {\GF }{\mathrm{GF}}
\newcommand {\Idm }{\mathrm{I}}
\newcommand {\Jordan }{\mathrm{J}}
\newcommand {\RevIdm }{\mathrm{R}}
\newcommand {\rad }{\operatorname{rad}}
\newcommand {\disc }{\operatorname{disc}}
\newcommand {\ind }{\operatorname{ind}}

\usepackage{ifdraft}
\ifdraft{\usepackage[outer]{showlabels}
	\usepackage{todonotes}}{}
	\ifdraft{\usepackage[outer]{showlabels}
		\usepackage{todonotes} \usepackage{listlbls} \usepackage{datetime}}{}
		
\usepackage[pagebackref,colorlinks,citecolor=red,urlcolor=blue,linkcolor=green, bookmarks=false,hypertexnames=true]{hyperref}
\usepackage{orcidlink}
\usepackage{fancyhdr}

\begin{document}

\bibliographystyle{plain}

\setcounter{page}{1}

\thispagestyle{empty}


\author{Klaus Nielsen}\,\orcidlink{0009-0002-7676-2944}
\email{klaus@nielsen-kiel.de}

\ifdraft{\today \ \currenttime}{\date{November 15, 2024}}
\pagestyle{fancy}
\fancyhf{}
\fancyhead[OC]{Klaus Nielsen}
\fancyhead[EC]{Products of symplectic involutions}
\fancyhead[OR]{\thepage}
\fancyhead[EL]{\thepage}

\maketitle

\begin{abstract}

We show that the symplectic groups $\SpG(2n, K)$ over a field $K$ are 5-reflectional  if $2n \ge 4$. If $n$ is even and $(2n, K) \ne (4, \GF(3))$ they are even 4-reflectional. 
\end{abstract}


\section{Introduction}

This is a continuation of our paper \cite{KNielsen-2025a}.
So we assume familiarity with the definitions and results of that paper.
We always assume that $K$ denotes a field of characteristic different from two.

We prove the following 2 theorems announced in \cite{KNielsen-2025a}.

\begin{theorem} \label{theorem-1}
	$\SpG(2n, K)$ is 4 reflectional if $n$ is even and $(n, |K|) \ne (2, 3)$.
\end{theorem}

\begin{theorem} \label{theorem-2}
 $\SpG(2n, K)$ is 5 reflectional.
\end{theorem}

In \cite[Proposition 4.1]{Pazzis-2024b}, C. de Seguins Pazzis has shown the following
result.

\begin{theorem}   
	Let $K$  be infinite. If $\SpG(2n, K)$ is $k$-reflectional, then $\SpG(2n+4, K)$ is $k$-reflectional.
\end{theorem}

We have a similar result:

\begin{theorem}   \label{theorem-3}
	 Assume that $\SpG(6, K)$ is 4-reflectional. Then
	\begin{enumerate}
		\item $\SpG(2n, K)$ is 4-reflectional for $2n \ge 14$.
		\item If $|K| \ge 9$, then $\SpG(10, K)$ is 4-reflectional.	
	\end{enumerate}
\end{theorem}

And as a byproduct of our observations, we have 

\begin{theorem} \label{theorem-4}
	Let $P \in \SpG(2n, K)$. 
	\begin{enumerate}
		\item  $P$ is the product of 2 big transvections if and only if $P$ has no elementary divisor $(x+1)^d$ of odd degree $d$.
		\item  $P$ is the product of 3 big transvections.
	\end{enumerate}
\end{theorem}
In the spirit of O'Meara, we say that a linear transformation $\varphi$ is
a big transvection if its  minimal polynomial is $\mu(\varphi) = (x-1)^2$.
[O'Meara \cite{Omeara-1978} calls a linear transformation a big dilation if its minimal polynomial is $(x-1)(x-\lambda)$, where $\lambda \ne 0, 1$.]

It is easy to show that a matrix $P \in \GL(n, K)$ is the product of 2 big transvections if and only if $P$ is similar to its inverse and has no elementary divisor $(x+1)^d$ of odd degree $d$. Further it is well known that a matrix with  determinant one is the product of 2 unipotent matrices, hence it is the product of 4 big transvections.

\section{Prelimanaries}

\begin{definition}
	Let $P \in \SpG(2n, K)$. We say that $P$ is
	\begin{enumerate}
		\item (type 1)-free if $P$ has no elementary divisor $(x \pm 1)^d$ off odd degree $d$
		\item (type 1+)-free if $P$ has no elementary divisor $(x + 1)^d$ off odd degree $d$
	\end{enumerate}
\end{definition}

We need the following 2 results from our paper
\cite{KNielsen-2025b}. See \cite[Lemma 2.14]{KNielsen-2025b} and \cite[Corollary 2.12]{KNielsen-2025b}.

\begin{proposition} \label{proposition-1}
	Let $P \in \SpG(2n, K)$ be orthogonally indecomposable of type 1 with minimal polynomial $\mu(P) = (x-1)^n$. Then  $P$ is conjugate to a matrix 
\[
 \left (\begin{array} {cc} \Idm_n & B  \\ C & \Idm_n +CB \end{array} \right ).
\]
\end{proposition}

\begin{proposition} \label{proposition-2}
	Let $P \in \SpG(2n, K)$. Then $P$ is conjugate to a matrix 
	\[
	P_1 := \left (\begin{array} {cc} 0 & B  \\ -B^{-1} & D \end{array} \right ),
	\]
	where $B$ is symmetric, if and only if $P$ is (type 1)-free. 
\end{proposition}

\begin{definition}
	Let
	\[
	A := \left (\begin{array} {cc} A_1 & A_2  \\ A_3 & A_4 \end{array} \right ),
	B := \left (\begin{array} {cc} B_1 & B_2  \\ B_3 & B_4 \end{array} \right ).
	\]
	The diagonal sum $A \boxplus B$ of the matrices $A$ and $B$ is the matrix
	\[
	\left (\begin{array} {cc} A_1 \oplus B_1 & A_2 \oplus B_2 \\ A_3 \oplus B_3 & A_4 \oplus B_4 \end{array} \right ).
	\]
\end{definition}

\begin{remark}
Let $P \in \SpG(2m, K), Q \in \SpG(2n, K)$. Then $P \boxplus Q \in \SpG(2m + 2n, K)$, 
and $P \boxplus Q$ is similar to $P \oplus Q$.
\end{remark}

\begin{proposition} \label{proposition-3}
Let $P \in \SpG(2n, K)$  be (type 1)-free. Then $P$ is conjugate to a matrix 
\[
P_1 := \left (\begin{array} {cc} \Idm_n & B  \\ C & D \end{array} \right ).
\]
\end{proposition}

\begin{proof}
	 Using  \ref{proposition-2}, we see that $P$ is conjugate to a matrix
	\[
	\left (\begin{array} {cc} \Idm_n & 0  \\ -B^{-1} &  \Idm_n\end{array} \right )
	 \left (\begin{array} {cc} 0 & B  \\ -B^{-1} & D \end{array} \right )
	\left (\begin{array} {cc} \Idm_n & 0  \\ B^{-1} &  \Idm_n\end{array} \right )
	= \left (\begin{array} {cc} \Idm_n & B  \\ (D-2 \Idm_n)B^{-1} &  D -\Idm_n\end{array} \right ).
	\]
\end{proof}

\begin{corollary} \label{cor-1}
	Let $P \in \SpG(2n, K)$. If  $P$ is (type 1+)-free, then $P$ is conjugate to a matrix 
	\[
	P_1 := \left (\begin{array} {cc} \Idm_n & B  \\ C & D \end{array} \right ).
	\]
\end{corollary}

We prove theorem \ref{theorem-4}:

\begin{proof}
	1: If $P$ is (type 1+)-free, then $P$ is similar to
	\[
	\left (\begin{array} {cc} \Idm_n & B  \\ C & D \end{array} \right )
	\left (\begin{array} {cc} \Idm_n & 0  \\ C &  \Idm_n\end{array} \right )
	\left (\begin{array} {cc} \Idm_n & B  \\ 0 &  \Idm_n\end{array} \right ).
	\]
	Conversely,  let $P =AB$ is the product of 2 big transvections $A$ and $B$.
	Then $AP -P^{-1}A = 2P -2\Idm_{2n} = BP -P^{-1}B$. Hence $A$ and $B$ commute with $P+P^{-1}$. It follows that the spaces $\ker (P+\Idm_{2n})^{2s}$ are $A$- and $B$-invariant. Suppose that $P$ has an elementary divisor $(x+1)^{2t+1}$. Then
	$PB^{-1}$ is not unipotent on the factor space $\ker (P+\Idm_{2n})^{2t+2}/\ker (P+\Idm_{2n})^{2t}$.
	
	2: Let $P$ be orthogonally indecomposable of type 1 with mininimal polynomial $(x+1)^n$. Let $U$ be the upper Jordan matrix with mininimal polynomial $(x+1)^n$. Then $P$  is  conjugate to a  cyclic matrix and a big transvection: 
	\[
	\left (\begin{array} {cc} U & 0  \\ 0 &  U^+\end{array} \right )
	= \left (\begin{array} {cc} U & -U  \\ 0 &  U^+\end{array} \right )
	 \left (\begin{array} {cc} \Idm_n & \Idm_n  \\ 0 &  \Idm_n \end{array} \right ). 
	\]
	Hence $P$ is the product of 3 big transvections.
\end{proof}

\section{Proof of theorem \ref{theorem-1}}

\begin{definition}
	Let $P \in \SpG(2n, K)$. We say that $P$ is strictly hyperbolic if 
	 $\mu(P) = qq^*$, where $q$ and and its reciprocal $q^*$ are coprime.
\end{definition}

Obviously, strictly hyperbolic transformations are hyperbolic.

\begin{lemma}                                    \label{lemma-0}
	Let 
	\[
	P = \left (\begin{array} {cc} A & 0  \\ 0 & A^+ \end{array} \right ) \in \GL(n, K),
	\]
	where $A \in \GL(n, K)$ is cyclic or bicyclic, and $\det A = \pm 1$. Then $P$ is 3-reflectional.
\end{lemma}

\begin{proof}
By \cite[THEOREM 2]{Ballantine-1977}, $A$ is 3-reflectional, hence $P$ is 3-reflectional.
\end{proof}

\begin{definition}
	Let 
	\[
	P = \left (\begin{array} {cc} A & B  \\ C & D \end{array} \right ) \in \GL(n, K),
	\]
	where $A \in \GL(m, K)$. The matrix $[P/A] : D - CA^{-1}B$ is called the Schur complement of $A$ in $P$.
\end{definition}

\begin{remark}
	Let 
	\[
	P = \left (\begin{array} {cc} A & B  \\ C & D \end{array} \right ) \in \SpG(2n, K),
	\]
	where $A \in \GL(n, K)$ is nonsingular. Then $[P/A] = A^+$.
\end{remark}

\begin{lemma}                                    \label{lemma-1}
	Let $\Omega$ be a strictly hyperbolic conjugacy class of $\SpG(2n, K)$. Then $\Omega^2$ contains all (type 1+)-free transformations of $\SpG(2n, K)$.
\end{lemma}

\begin{proof} 
	Let $\mu(\Omega) = q q^*$, where $q$ and its reciprocal $q^*$ are coprimary.
	Let $P \in \SpG(2n, K)$. Assume that $P$ has no elementary divisor $(x + 1)^{2t+1}$ of odd degree. By proposition \ref{proposition-3}, $P$ is conjugate to a matrix 
	\[
	P_1 := \left (\begin{array} {cc} \Idm_n & B  \\ C & D \end{array} \right )
	= \left (\begin{array} {cc} A & 0  \\ CA^{-1} & A^+ \end{array} \right )
	\left (\begin{array} {cc} A^{-1} & A^{-1}B  \\ 0 & A^+ \end{array} \right ).
	\]
	Now $P \in \Omega^2$ for a suitable matrix $A$ with $\mu(A) = q$.
\end{proof}

 \begin{corollary}
 	Let $K \ne \GF(3)$ and $n$ be even. Then $\SpG(2n, K)$ is 4 reflectional.
\end{corollary}

\begin{proof}
	Let $P \in \SpG(2n, K)$.  If $P$ is orthogonally indecomposable of type 1, then
	$P$ is bireflectional by proposition \cite[Theorem 1.2]{KNielsen-2025a}. So by induction, we may assume that $P$  has at most one orthogonally indecomposable orthogonal summand of type 1. Clearly, $P$ is 4-reflectional if $-P$ is 4-reflectional. So we may assume that $P$ is (type 1+)-free.
	 By lemma \ref{lemma-1}, $P = H_1 H_2$ is the product of 2 semisimple strictly hyperbolic transformations $H_1, H_2$ with minimal polynomials $\mu(H_1) = \mu(H_2) = (x-\lambda)(x-\lambda^{-1})$, where $\lambda^2 \ne 1$.
	By proposition \ref{proposition-1}, $H_1$ and $H_2$ are bireflectional,  hence $P$ is 4-reflectional.
\end{proof}

 \begin{lemma}   \label{lemma-2}
 	Let $P \in \SpG(4,3)$ be orthogonally indecomposable. Then $P$ is 4-reflectional.
\end{lemma}
	
 \begin{proof}
Assume first that  $P$ has minimal polynomial $x^4+1 = (x^2 + x -1)(x^2 - x -1)$ or $(x^2 +1)^2$ or $(x \pm 1)^4$. Then $P$ is conjugate to a matrix
\[
\left (\begin{array} {cc} A & B\\ 0 & A^+ \end{array} \right ).
\]
In any case, there exists an involution $S \in \GL(2,3)$ such that
\[
\left (\begin{array} {cc} S & 0\\ 0 & S^+ \end{array} \right )
\left (\begin{array} {cc} A & B\\ 0 & A^+ \end{array} \right ).
\]
has minimal polynomial $x^2+1$ or $(x^2 + x -1)(x^2 - x -1)$. 
So $P$ is 4- reflectional.   If $\mu(P) = x^4+1$, then $P$ is even 3-reflectional.

It remains to consider the case $\mu(P) = x^4 \pm x^3 + x^2 \pm x +1$.
By Wall \cite[p. 36]{Wall}, every
symplectic transformation similar to $P$ is already conjugate to $P$.
Hence $P$ is conjugate to a matrix $Q$, where  
\[
Q = \left (\begin{array} {cc} 0 & \Idm_2\\ -\Idm_2 & D \end{array} \right ),
D = \left (\begin{array} {cc} 0 & 1\\ 1 & \pm 1 \end{array} \right ).
\]
We have $\mu(D) = x^2 \pm x-1$ and $D^4 = -\Idm_2$. Now $Q$ is a product of 2 skew-involutions.
(Actually, this can easily be shown without using Wall's theorem):
\[
Q = \left (\begin{array} {cc} 0 & \Idm_2\\ -\Idm_2 & D \end{array} \right )
= \left (\begin{array} {cc} -D^2 & 0\\ -D^3 & D^2 \end{array} \right )
\left (\begin{array} {cc} 0 & D^2\\ D^2 & 0 \end{array} \right ).
\]
\end{proof}

\begin{lemma}      \label{lemma-3}
	Let $P \in \SpG(4, K)$ with minimal polynomial $(x+1)(x-1)^2$. Then $P$ is the product of an involution and a  hyperbolic big transvection. In particular, $P$ is 3-reflectional.
\end{lemma}

\begin{proof}
	
	$P$ is conjugate to
	\[
	\left (\begin{array} {cccc} 1 & 0 & 0 & -1 \\ 0 & 1 & -1 & 0  \\ 0 & 0 & 1 & 0 \\ 0 & 0 & 0 & 1 \end{array} \right )
	\left (\begin{array} {cccc} 1 & 0 & \lambda & 0 \\ 0 & -1 & 0 & 0  \\ 0 & 0 & 1 & 0 \\ 0 & 0 & 0 & -1 \end{array} \right )
	\left (\begin{array} {cccc} 1 & 0 & 0 & 1 \\ 0 & 1 & 1 & 0  \\ 0 & 0 & 1 & 0 \\ 0 & 0 & 0 & 1 \end{array} \right )
	\]
	\[
	= \left (\begin{array} {cccc} 1 & 0 & \lambda & 2 \\ 0 & -1 & -2 & 0  \\ 0 & 0 & 1 & 0 \\ 0 & 0 & 0 & -1 \end{array} \right ) =
	\left (\begin{array} {cccc} 1 & 0 & 0 & 0 \\ 0 & -1 & 0 & 0  \\ 0 & 0 & 1 & 0 \\ 0 & 0 & 0 & -1\end{array} \right )
	\left (\begin{array} {cccc} 1 & 0 & \lambda & 2 \\ 0 & 1 & 2 & 0  \\ 0 & 0 & 1 & 0 \\ 0 & 0 & 0 & 1\end{array} \right ).
	\]
\end{proof}

\begin{lemma}      \label{lemma-4}
	Let $P \in \SpG(4, 3)$ with minimal polynomial $(x-1)(x^2+1)$. Then $P$ is the product of 2 skew-involutions. In particular, $P$ is 4-reflectional.
\end{lemma}

\begin{lemma}   \label{lemma-4a}
	$\SpG(2n, 3)$ is 4 reflectional if $n$ is even and $n \ge 4$.
\end{lemma}

\begin{proof}
	Let $P \in \SpG(2n, 3)$.  Assume first that $P$ is (type 1+)-free.
	 By lemma \ref{lemma-1}, $P$ is the product of 2 semisimple strictly hyperbolic transformations with chatacteristic polynomial $(x^2-x-1)^{2t} (x^2+x-1)^{2t}$ if $n =4t \equiv 0 \mod 4$ and
	 $(x^3+x^2 -1)^2 (x^3 -x  -1)^2 (x^2-x-1)^{2t} (x^2+x-1)^{2t}$ if $n =4t + 6 \equiv 2 \mod 4$. As in the proof of \ref{lemma-1}, $P$ is 4-reflectional.
	 It follows that $P$ is 4-reflectional if $-P$ is (type 1+)-free as $P$ is 4-reflectional
	 if $-P$ is 4-reflectional.
	 
	 So by induction, we may assume that $P = A \boxplus B \boxplus C$, where 
	 $A$ and $B$ are orthogonally indecomposable of type 1, and $C \in  \SpG(4n, 3)$ is not 4-reflectional. By lemma \ref{lemma-2}, $C = C_1 \boxplus C_2$, where $\dim C_1 = 2$. 
	 As just shown, $A \boxplus C_1$ is 4-reflectional if $\dim A \ge 6$. And if $A = \pm \Idm_2$, then $A \boxplus C_1$ is 4-reflectional by lemmata \ref{lemma-3} and \ref{lemma-4}. Similarly, $B \boxplus C_2$ is 4-reflectional. Hence $P$ is 4-reflectional.
\end{proof}

\section{Proof of theorem \ref{theorem-2}}

Our old proof only works for finite fields or other fields with small u-invariant, e.g. algebraicly closed fields.

Let $(V, f)$ be a nondegenerate symplectic vector space over $K$.
Let $\varphi \in \SpG(V)$.
For $v \in V$ let  $Q_{\varphi}(v) = f(v\varphi, v)$. Then $Q_{\varphi}$ is a quadratic form with polar form $f_Q$, where 
$f_Q(u,w) = f(u \varphi, w) + f(w \varphi, u)$. We have $\rad Q_{\varphi} = \Fix(\varphi^2)$.
We need the following auxilary result.

\begin{lemma}     \label{lemma-5}
	Let $\ch K \ne 2$.  Let $(V, Q)$ be a nondegenerate quadratic vector space over $K$ with Witt index $\ind  \ge 1$. If the set of all isotropic vectors of $V$ is 
	contained in the union of two hyperplanes $H_1, H_2$ of $V$, then 
	\begin{enumerate}
		\item $(V, Q)$ is a hyperbolic plane or
		\item $K = \GF(3)$, $\dim V = 3$,
		and $H_1$ and $H_2$ are hyperbolic planes.
	\end{enumerate}
\end{lemma}  

We omit the simple proof.

\begin{lemma}  \label{lemma-6}
	Let $(V, f)$ be a nondegenerate symplectic vector space over $K$.
	Assume that $\dim V \ge 6$.
	Let $\varphi \in \SpG(V, f)$.
	Assume that  $Q_{\varphi}$ has Witt index $\ge 1$. If $\rad Q_{\varphi} = 0$ assume even that
	$Q_{\varphi}$ has Witt index $\ge 2$. Then there exists an involution $\sigma \in 
	\SpG(V, f)$ such that $\Fix(\phi \sigma)$ is not totally degenerate.
\end{lemma}

\begin{proof}
	Assume first that $\Fix(\varphi^2) = 0$. 
	Let $v \in V$ be $Q_{\varphi}$-isotropic.  By lemma \ref{lemma-5}, there exists a $Q_{\varphi}$-isotropic vector $w  \in v(\varphi - \varphi^{-1})^{\perp} - [v(1-\varphi)^{\perp} \cup v^{\perp}]$.

	The planes $S := \langle v \varphi -v, w\varphi -w \rangle$ and $T := \langle v, w \rangle$ are  nondegenerate, and $v \varphi + v, w \varphi + w \in S^{\perp}$. Let $\sigma \in \SpG(V, f)$ be the involution with $\Fix(\sigma) = S^{\perp}$. Then $T \le \Fix(\varphi \sigma)$. 
	
	If $\dim \Fix(\varphi^2) \ge \frac{\dim V}{2} + 1$, then already $\Fix(\varphi^2)$  is not totally degenerate.
	So let $0 \ne \Fix(\varphi^2) \le \frac{\dim V}{2}$. We may assume that
	$\Fix(\varphi) \ne 0$. Let $x \in \Fix(\varphi)$ be nonzero, and let $y \in V -x^{\perp}$ be $Q_{\varphi}$-isotropic. Then $T := \langle x,y\varphi+y\rangle$ is nondegenerate and $y - y\varphi \in T^{\perp}$.
	Let $\tau \in \SpG(V, f)$ be the involution with $\Fix(\tau) = T$. Then $\Fix(\varphi \tau)$ contains the nondegenerate plane $\langle x,y \rangle$. 
\end{proof}

\begin{corollary} 
	Let $K$ be finite. If $K = \GF(3)$ let $2n \ge 8$. Then 
	$\SpG(2n, K)$ is 5-reflectional.
\end{corollary}

Next we consider infinite fields. By theorem \ref{theorem-1} and \cite[Theorem 1.2]{KNielsen-2025a}, it suffices to show that (type 1)-free transformations are 5-reflectional.

\begin{lemma}                         \label{lemma-7}
	Let $P \in \SpG(2n, K)$ be orthogonally indecomposable of type 1 with minimal polynomial $(x-1)^n$. Let $S \in \GL(n, K)$ be  a symmetric, upper antitriangular matrix with
	$\sum_{j=1}^n (-1)^j S_{j,n+1-j} = 0$.
	Let $U = \Jordan_n(1)$ be the upper Jordan matrix with eigenvalue 1.
	Then $P$ is conjugate to 
	\[
	Q := \left (\begin{array} {cc} 0  & US\\ -(US)^+ & U^+ +U^S  \end{array} \right ).
	\]
\end{lemma}

\begin{proof}
	Let 
	\[
	M = \left (\begin{array} {cc} U  & US\\ 0 & U^+  \end{array} \right ).
	\]
	Then $M \in  \SpG (2n, K)$. 
	There exists an  involution $T \in \GL(n, K)$  inverting $U$. Clearly, $T$ 
	must be upper triangular, and $T_{j,j}= - T_{j+1, j+1}$. We have  
	$\RevIdm_n (U') \RevIdm_n  = U$, where $\RevIdm_n$  is the antidiagonal identity matrix of $\GL(n, K)$.  Hence $M$ is similar to 
	\[
	\left (\begin{array} {cc} \Idm_n  & 0 \\ 0 &  T \RevIdm_n   \end{array} \right )
	\left (\begin{array} {cc} U  & US \\ 0 &  U^+ \end{array} \right )
	\left (\begin{array} {cc} \Idm_n  & 0 \\ 0 &  \RevIdm_n T  \end{array} \right )
	=\left (\begin{array} {cc} U  & US \RevIdm_n T\\ 0 & U  \end{array} \right ). 
	\]
	Now $S \RevIdm_n$ is upper triangular. Hence $U S \RevIdm_n T$
	is  upper triangular with zero trace.
	By \cite[Lemma 3]{Johnson_Schreiner-1991}, $M$ 
	is similar  to $U \oplus U$. Hence $P$ is conjugate to $M$.
	Now
	\[
	\left (\begin{array} {cc} \Idm_n  & 0\\ S^{-1} &  \Idm_n  \end{array} \right )
	\left (\begin{array} {cc} U  & US\\ 0 & U^+  \end{array} \right )
	\left (\begin{array} {cc} \Idm_n  & 0\\ -S^{-1} &  \Idm_n  \end{array} \right )
	= \left (\begin{array} {cc} 0  & US\\ -(US)^+ & U^+ +U^S  \end{array} \right ).
	\]
\end{proof}

\begin{lemma} \label{lemma-7a}
	Let $2n \ge 6$ and $2n \equiv 2 \mod 4$.  Let $K$ be infinite. 
	Let $P \in \SpG(2n,K)$ be (type 1)-free. 
	Then $P = Z Q$ is the product of a 3-reflectional matrix $Z$ and   an orthogonally indecomposable matrix $Q \in \SpG(2n,K)$ of type 1. 
\end{lemma}

\begin{proof}
	By proposition \ref{proposition-2}, we may assume that
	\[
	P = \left (\begin{array} {cc} A & -B^{-1}\\ B & 0 \end{array} \right),
	\]
	where $B$ is symmetric. Let  $Q$ is orthogonally indecomposable of type 1. It follows from lemma \ref{lemma-7} that $Q$ is conjugate to a matrix 
	\[
	\widetilde{Q} = \left (\begin{array} {cc} 0 & C\\ -C^+ & D \end{array} \right),
	\]
	where $C$ is upper triangular and has discriminant $\disc C = \disc B$.
	There exists a cyclic matrix $E := BXCX'$  with determinant
	one such that $E$ and its inverse are coprimary.
 Hence 
	\[
	\left (\begin{array} {cc} A & -B^{-1}\\ B & 0 \end{array} \right)
	\left (\begin{array} {cc} X & 0\\ 0 & X^+ \end{array} \right)
	\left (\begin{array} {cc} 0 & C\\ -C^+ & D \end{array} \right)
	\left (\begin{array} {cc} X^{-1} & 0\\ 0 & X' \end{array} \right)
	= \left (\begin{array} {cc} E^+ & \ast\\ 0 & E \end{array} \right)
	\]
	is 3-reflectional by lemma \ref{lemma-0}.
\end{proof}

It remains to consider the case $K = \GF(3), 2n \le 6$.

\begin{lemma} \label{lemma-8}
Let $P \in \SpG(6,3)$ be (type 1)-free. 
Then $P = ZQ$, where $Z$ is cyclic with minimal polynomial $(x^3+x^2-1)(x^3-x-1)$ 
and $Q$ is orthogonally indecomposable of type 1.
\end{lemma}

\begin{proof}
Let 
\[
U = \left (\begin{array} {ccc} 1  & 1 & 0\\ 0 & 1 & 1  \\ 0 & 0 & 1\end{array} \right ),
S = \left (\begin{array} {ccc} 0  & 0 & 1\\ 0 & -1 & 0  \\ 1 & 0 & 0 \end{array} \right ),
Q = \left (\begin{array} {cc} 0  & US\\ -(US)^+ & U^+ + U^S  \end{array} \right ), 
\]
By lemma \ref{lemma-7}, $Q$ is orthogonally indecomposable of type 1. Now $US$ has minimal polynomial $x^3+x^2 -1$.
By proposition \ref{proposition-2}, we may assume that
\[
P = \left (\begin{array} {cc} A &  B\\ -B^{-1} & 0 \end{array} \right).
\]
As $K$ is finite, $B$ is congruent to $\Idm_3$ or $-\Idm_3$. So we may assume 
that $B = -\Idm_3$.
Hence  
\[
PQ = \left (\begin{array} {cc} A & -\Idm_3 \\ \Idm_3 & 0 \end{array} \right)
\left (\begin{array} {cc} 0  & US \\ -(US)^+ & U^+ + U^S  \end{array} \right )
= \left (\begin{array} {cc} (US)^+  & \ast\\ 0& US \end{array} \right )
\]
has minimal polynomial $(x^3+x^2-1)(x^3-x-1)$, and $PQ$ is 3-reflectional by \ref{lemma-0}.
\end{proof}

\begin{lemma}  \label{lemma-8a}
	$\SpG(4, 3)$ is 5-reflectional.
\end{lemma}	

\begin{proof}
Let
\[
U = \left (\begin{array} {cc} 1 & 1\\ 0 & 1 \end{array} \right ),
X = \left (\begin{array} {cc} 0 & 1\\ -1 & 0 \end{array} \right ).
\]
Clearly $P \in \SpG(4, 3)$ is k-reflectional if $-P$ or $P^{-1}$ is k-reflectional.
By \cite[Theorem 1.2]{KNielsen-2025a}, $U \boxplus U^{-1}$ and $X\boxplus X$ are bireflectional.
		
By lemmata \ref{lemma-2}, \ref{lemma-3}, and \ref{lemma-4},
we are  left with the matrices  $-U \boxplus U^{-1}$, $\Idm_2 \boxplus U$, $U \boxplus U$, $-U \boxplus U$, and $U \boxplus X$.
		 
Clearly, $-U \boxplus U^{-1}$ is 3-reflectional, and 
$\Idm_2 \boxplus U$ is 4-reflectional as $-\Idm_2 \boxplus U$ is 3-reflectional.
The matrix $U \boxplus U$ is a product of 2 skew-involutions: 
\[
\left (\begin{array} {cc} \Idm_2 & \Idm_2\\ 0 & \Idm_2 \end{array} \right )
	 = \left (\begin{array} {cc} X & -X\\ 0 & -X \end{array} \right )
\left (\begin{array} {cc} -X & 0\\ 0 & X \end{array} \right ).
\]
 Hence $U \boxplus U$ is 4-reflectional, and $-U \boxplus U$ is 5-reflectional.
Finally, $U \boxplus X$ is conjugate to $(U \boxplus U^{-1} ) (\Idm_2 \boxplus -U^+)$.
It follows that $U \boxplus X$ is 5-reflectional.
\end{proof}	

\subsection{Alternative proof of theorem \ref{theorem-2}}

\begin{lemma} \label{lemma-9}
	Let $n = 2m+3 \ge 5$.  Let $P \in \SpG(2n,3)$ be (type 1)-free. Then $P = R Q$ is the product of 2 symplectic matrices $R$ and $Q  = Q_1 \boxplus Q_2$, where
	\begin{enumerate}
		\item $Q_1 \in \SpG(6,3)$ is orthogonally indecomposable of type 1,
		\item $\mu(Q_2) = x^2+1$,
		\item  $R$ is semisimple with characteristic polynomial $\chi(R)= (x^3 +x^2-1)(x^3-x-1)(x^2+x-1)^m (x^2-x-1)^m$.
	\end{enumerate}
\end{lemma}

\begin{proof}
	As above, we may assume that
	\[
	P = \left (\begin{array} {cc} A &  -\Idm_n\\ \Idm_n & 0 \end{array} \right).
	\]
	Let 
	\[
	Q_1 = \left (\begin{array} {cc} 0  & US\\ -{US}^+ & D  \end{array} \right ),
	T = \left (\begin{array} {cc} 0  & \Idm_m\\ \Idm_m & \Idm_m  \end{array} \right ),
	Q_2 = \left (\begin{array} {cc} 0  & T\\ -T^{-1} & 0  \end{array} \right ), 
	\]
	where $Q_1 \in \SpG(6,3) $ is orthogonally indecomposable of type 1 as in lemma \ref{lemma-7}. We have $\mu(Q_2) = x^2+1$. Put 
	\[
	R = P (Q_1 \boxplus Q_2 )=  \left (\begin{array} {cc} {US}^+ \oplus T^{-1}   & \ast\\ 0 &  {US} \oplus T\end{array} \right ).
	\]
	and $Q = (Q_1 \boxplus Q_2 )^{-1}$. We are done.
\end{proof}

\begin{lemma} \label{lemma-10}
	Let $2n \ge 6$ and $2n = 4m + 2\equiv 2 \mod 4$.  Let $|K| \ge 4$. Let $P \in \SpG(2n,K)$ be (type 1)-free. 
	Then there exists a semisimple matrix $Q \in \SpG(2n,K)$ with characteristic polynomial $(x-1)^2 (x-\lambda)^{2m} (x-\lambda^{-1})^{2m}$, where $\lambda^2 \ne 1$, such that $PQ$ is 3-reflectional.
\end{lemma}

\begin{proof}
	Let $\lambda \in K^*$ with $\lambda^2 \ne 1$. Let
	\[
	U = \left (\begin{array} {cc} 0 & 1\\ \lambda & 0 \end{array} \right),
	X = \left (\begin{array} {cc} 0 & U\\ -U^+ & \Idm_2 + U^+ U \end{array} \right),
	\]
	\[
	Y = \left (\begin{array} {cc} 0 & S\\ -S^{-1} & (\lambda + \lambda^{-1})\Idm_{n-2}  \end{array} \right).
	\]
	Then $\mu(X) = (x-1)(x-\lambda)(x-\lambda^{-1})$ and $\mu(Y) = (x-\lambda)(x-\lambda^{-1})$.
	Let 
	\[
	Q = X \boxplus Y = \left (\begin{array} {cccc} 0 & 0 & U & 0\\  0 & 0 & 0 & S\\
		-U^+ & 0 & D_1 & 0\\ 
		0  & -S^{-1} & 0 & D_4\\ \end{array} \right).
	\]
	Then $\mu(Q) = (x-1)(x-\lambda)(x-\lambda^{-1})$ and $Q$ is bireflectional. 
	We may assume that
	\[
	P = \left (\begin{array} {cccc} A_1 & A_2 & -C_1^{-1} & 0\\  A_3 & A_4 & 0 & -C_2^{-1}\\
		C_1 & 0 & 0 & 0\\ 
		0  & C_2 & 0 & 0\\ \end{array} \right),
	\]
	where $C_1 \in \GL(2, K)$ and $C_2 \in \GL(n-2, K)$ are symmetric. 
	If $K$ is finite, we may aditionally assume that $C_1 = \Idm_2$. Now for a suitable 
	$\lambda \in K$ and symmetric matrix $S$, the matrix
	$E := C_1 U \oplus C_2 S$ is semisimple with 
	characteristic polynomial $(x^2+\mu) (x+\mu^{-1})^{\frac{n-3}{2}} (x-\mu)^{\frac{n-3}{2}}$, where $\mu^2 \ne 1$.
	Again by \ref{lemma-0}, $E$ is 3-reflectional. Hence 
	\[
	P Q = \left (\begin{array} {cccc} A_1 & A_2 & -C_1^{-1} & 0\\  A_3 & A_4 & 0 & -C_2^{-1}\\
		C_1 & 0 & 0 & 0\\ 
		0  & C_2 & 0 & 0\\ \end{array} \right)
	\left (\begin{array} {cccc} 0 & 0 & U & 0\\  0 & 0 & 0 & S\\
		-U^+ & 0 & D_1 & 0\\ 
		0  & -S^{-1} & 0 & D_4\\ \end{array} \right)
	\]	
	\[	
	=  \left (\begin{array} {cccc} (C_1 U)^+ & 0 & \ast & \ast\\  0 & (C_2 S)^+ & \ast& \ast\\
		0 & 0 & C_1 U & 0\\ 
		0  & 0 & 0 & C_2 S\\ \end{array} \right)
	\]
	is 3-reflectional, as $E$ and its inverse are coprimary.
\end{proof}

\section{Proof of theorem \ref{theorem-3}}

\begin{lemma}   \label{lemma-11}
	Let $P \in \SpG(4, 3)$ with $\mu(P) = (x-1)(x^2+1)$. There exist involutions 
	$S_1, S_2 \in \SpG(4,3)$ such that $\mu(PS_1) = (x^2+x-1) (x^2-x-1)$ and  
	$\mu(PS_2) = (x^2+1) (x-1)^2$.
\end{lemma}

\begin{proof}
	Let
	\[
	X = \left (\begin{array} {cc} 0 & -1\\ 1 & 0 \end{array} \right),
	U = \left (\begin{array} {cc} -1 & 0\\ 1 & 1 \end{array} \right),
	S_1 = \left (\begin{array} {cc} X & X\\ X & -X \end{array} \right),
	S_2 = \left (\begin{array} {cc} U & 0\\ 0 & U' \end{array} \right).
	\]
	Then 
	\[
	[X \boxplus \Idm_2] S_1 = \left (\begin{array} {cccc} 0 & 0 & -1 & 0\\  0 & 1 & 0 & 0\\ 
		1 & 0 & 0 & 0\\ 0 & 0 & 0 & 1
	\end{array} \right)
	\left (\begin{array} {cccc} 0 & -1 & 0 & -1\\  1 & 0 & 1 & 0\\ 
		0 & -1 & 0 & 1\\ 1 & 0 & -1 & 0
	\end{array} \right)
	= \left (\begin{array} {cccc} 0 & 1 & 0 & -1\\  1 & 0 & 1 & 0\\ 
		0 & -1 & 0 & -1\\ 1 & 0 & -1 & 0
	\end{array} \right),	 
	\]
	\[
[X \boxplus \Idm_2] S_2 = \left (\begin{array} {cccc} 0 & 0 & -1 & 0\\  0 & 1 & 0 & 0\\ 
		1 & 0 & 0 & 0\\ 0 & 0 & 0 & 1
	\end{array} \right)
	\left (\begin{array} {cccc} -1 & 0 & 0 & 0\\  1 & 1 & 0 & 0\\ 
		0 & 0 & -1 & 1\\ 0 & 0 & 0 & 1
	\end{array} \right)
	= \left (\begin{array} {cccc} 0 & 0 & 1 & -1\\  1 & 1 & 0 & 0\\ 
		-1 & 0 & 0 & 0\\ 0 & 0 & 0 & 1
	\end{array} \right).	 
	\]	
\end{proof}

\begin{lemma}   \label{lemma-12}
	Let $2n \ge 10$.
	Let $P \in \SpG(2n, 3)$. Assume that  $P$ has an elementary divisor $(x \pm 1)^{2t+1}$ of odd degree. Then $P$ is 4-reflectional.
\end{lemma}

\begin{proof}
	We may assume that 
	$P = U  \boxplus Q$, where $U$ is orthogonally indecomposable of type 1 and $Q \in \SpG(4,3)$ is not 4-reflectional.
	
	By  lemma \ref{lemma-3},  $Q = A \boxplus B$, where $A, B \in \SpG(2, 3)$.
	Furthermore, we may asume that $A, B \ne \pm \Idm_2$ and $\mu(Q) \ne x^2+1$.
	So we may assume that $\mu(B) = (x-1)^2$.
	Let 
	\[
	U = \left (\begin{array} {cc} U_{n-2} & 0\\ 0 & U_{n-2}^+ \end{array} \right).
	\]
	Assume first that  $\mu(A) = x^2+1$.
	It follows from  \cite[Example 4.4]{KNielsen-2025a} that 
	$U_{n-2}$ is the product of an involution and a cyclic matrix 
	with minimal polynomial
	$(x+1)(x-1)^{n-5}(x^2+x-1)$. Hence $U$ is the product of an involution and a matrix $W$ conjugate to  $-\Idm_2 \boxplus Y \boxplus
	\widetilde{U}$, where $\widetilde{U}$ is unipotent and hyperbolic, and $\mu(Y) = (x^2-x -1)(x^2+x-1$. 
	By lemma \ref{lemma-3}, $-\Idm_2 \boxplus B$ is 3-reflectional.
	And it follows from \ref{lemma-11} that $Y \boxplus A$ is 3-reflectional.
	Hence $W$ is 3-reflectional, as $\widetilde{U}$ is even bireflectional.
	
	So let  $\mu(A) = (x \pm 1)^2$. (In fact $\mu(A) = (x + 1)^2$; see the proof of \ref{lemma-8a})
	Again using  \cite[Example 4.4]{KNielsen-2025a}, we see that 
	$U_{n-2}$ is the product of an involution and a cyclic matrix 
	with minimal polynomial
	$(x+1)(x-1)^{n-5}(x^2+1)$. Thus $U$ is the product of an involution a matrix $W$ conjugate to  $-\Idm_2 \boxplus X  \boxplus X \boxplus
	\widetilde{U}$, where $\widetilde{U}$ is unipotent and hyperbolic, and $\mu(X) = x^2 +1$. 
	And again  it follows from \ref{lemma-3} and \ref{lemma-11} that 
	$-\Idm_2 \boxplus B$ and $X  \boxplus X \boxplus A$ are 3-reflectional.	
\end{proof}

\begin{lemma}     \label{lemma-13}
Let 
\[
P = \left (\begin{array} {cc} A & 0 \\ 0 & A^+  \end{array} \right ) \in \SpG(2n,K),
\]
where $A \sim -A$. Then
\begin{enumerate}
   \item $P$ is the product of a skew-involution and a symplectic involution.
    \item $P$ is the product of a strictly hyperbolic transformation with minimal polynomial $(x-\lambda)(x-\lambda^{-1})$ and a symplectic involution if $|K| \ge 5$.
 \end{enumerate}
\end{lemma}

\begin{proof}
	It is known that $A = ST$ is the product of an antisymmetric and a symmetric matrix.
	And it is even well known that a matrix is the product of 2 symmetric matrices.
	These  results can e.g. be found in the papers of Stenzel \cite{Stenzel1922} or
	Shoda \cite{Shoda1929}. We give a short proof of the first fact using the second one. 
	Clearly, $n = 2m$ is even, and $A$ is similar to a matrix
	\[
	\left (\begin{array} {cc} 0 & \Idm_m \\ C & 0  \end{array} \right )
	= \left (\begin{array} {cc} 0 & -Q \\ Q & 0  \end{array} \right )
	\left (\begin{array} {cc} R & 0 \\ 0 & -Q^{-1}  \end{array} \right )
	\]
	 where $Q$ and $R$ are symmetric. Now
	\[
	\left (\begin{array} {cc} A & 0 \\ 0 & A^+  \end{array} \right ) =
	\left (\begin{array} {cc} 0 & S \\ -S^+ & 0  \end{array} \right )
	\left (\begin{array} {cc} 0 & -T^+ \\ T & 0  \end{array} \right )
	\]
	is the product of a symplectic skew-involution and a symplectic involution.
	And if $D = (\lambda + \lambda^{-1})\Idm_n$, then 
	$P$ is the product of a strictly hyperbolic transformation with minimal polynomial $(x-\lambda)(x-\lambda^{-1})$ and a symplectic involution:
	\[
	\left (\begin{array} {cc} A & 0 \\ DT & A^+  \end{array} \right ) =
	\left (\begin{array} {cc} 0 & S \\ -S^+ & D  \end{array} \right )
	\left (\begin{array} {cc} 0 & -T^+ \\ T & 0  \end{array} \right ).
	\]	
\end{proof}

We can now prove our last theorem.

\begin{proof}[Proof of theorem \ref{theorem-3}]
	
	Let $P \in \SpG(2n,K)$. By  theorem \ref{theorem-1}, lemma \ref{lemma-12}, and \cite[Theorem 1.2]{KNielsen-2025a}, we may assume that  $n$ is odd, and $P$ is (type 1)-free. Hence $P$ is similar to a matrix
	\[
	Q = \left (\begin{array} {cccc} 0 & 0 & B_1 & 0\\  0 & 0 & 0 & B_2\\
		-B_1^{-1} & 0 & D_1 & D_2\\ 
		0  & -B_2^{-1} & D_3 & D_4\\ \end{array} \right),
	\]
	where $B_1 \in \GL(n-3, K)$ and $B_2 \in \GL(3, K)$ are symmetric. Let 
	\[
	M = \left (\begin{array} {cc} 0 & B_2\\ -B_2^{-1} & D_4 \end{array} \right),
	X = \left (\begin{array} {cc} 0 & H\\ H^{-1} & 0 \end{array} \right),
	Y = \left (\begin{array} {cc} Y_1 & Y_2\\ Y_3 & Y_4 \end{array} \right),
	\]
	where $H \in \GL(n-3,K)$ is antisymmetric, and $Y \in \SpG(6, K)$ is an involution.
	We have $M \in \SpG(6, K)$. So we can assume that $YM$ is 3-reflectional.
	Let $S = X \boxplus Y$. Then $S \in \SpG(2n, K)$ is involutory, and
	\[
	SQ = 
	\left (\begin{array} {cccc} 0 & 0 & H & 0\\  0 & Y_1 & 0 & Y_2\\
		H^{-1} & 0 & 0 & 0\\ 0  & Y_3 & 0 & Y_4\\ \end{array} \right)
	\left (\begin{array} {cccc} 0 & 0 & B_1 & 0\\  0 & 0 & 0 & B_2\\
		-B_1^{-1} & 0 & D_1 & D_2\\ 
		0  & -B_2^{-1} & D_3 & D_4\\ \end{array} \right)
	\]
	\[
	= \left (\begin{array} {cccc} -HB_1^{-1} & 0 & H D_1 & H D_2\\  0 & -Y_2 B_2^{-1} & Y_2 D_3 & Y_1 B_2 + Y_2 D_4\\
		0 & 0 & H^{-1} B_1 & 0\\ 
		0  & -Y_4 B_2^{-1} & Y_3 B_1 + Y_4 D_3 & Y_3 B_2 + Y_4 D_4\\ \end{array} \right).
	\]
	We can achieve that $\mu(-H B_1^{-1})$ =
	\begin{enumerate}
		\item $(x^4 -x^2-1)^t$ if $q = 7$ and $n-3 = 4t$,
		\item $(x^4 -x^2-1)^{t-1} (x^6+x^2+1)$ if $q = 7$ and $n-3 = 4t+2$,
		\item  $(x^4+2)^t$ if $q = 5$ and $n-3 = 4t$,
		\item $(x^4+2)^{t-1}(x^6-x^2+2)$ if $q = 5$ and $n-3 = 4t + 2$,
		\item $(x^4 -x^2-1)^t$ if $q = 3$ and $n-3 = 4t$,
		\item $(x^4 -x^2-1)^{t-1} (x^6-x^2+1)$ if $q = 3$ and $n-3 = 4t + 2 \ge 18$.
	\end{enumerate}
	(All occuring polynomials of degree 4 and 6 are irreducible).
	In any case $HB_1$ and its inverse are coprimary. Further, $HB_1$ and $YM$ are coprimary.
	If $q \ge 9$ and $2n = 4t + 6 \ge 10$, we can achieve that $\mu(H B_1) = (x^2 + \lambda)^t$, where $\lambda^2 \ne 1$ and  $x^2+ \lambda$ is prime to $\mu(YM)$.
	By lemma \ref{lemma-13}, 
		\[
	\left (\begin{array} {cc} -H B_1^{-1} & 0\\ 0 & -H^+ B_1 \end{array} \right)
	\]
	is 3-reflectional, hence $SQ$ is 3-reflectional.
\end{proof}


\section{Final remarks}

Without proof we also mention the following result:

\begin{remark} \label{remark}  
	Let $\Omega$ be the  set of noncentral involutions in $\SpG(4,3)$.
	Let $P \in \SpG(4,3)$. Then
	\begin{enumerate}
		\item $P \in  \Omega^6$.
		\item $P \not \in   \Omega^5$ if and only if $P$ or $-P$ is a transvection.
		\item $P \in  \Omega^3 \cap \Omega^4$ if and only if 
		\begin{enumerate}
			\item $\mu_P = x^2 \pm 1$ or
		    \item $\mu_P = x^4+1$ or 
		    \item $\mu_P = (x - 1)^2(x + 1)^2$ and $P$ and $-P^{-1}$ are conjugate
       \end{enumerate}  
		\item $P \in  \Omega^3 -\Omega^4$ if and only if 
		$\mu_P = (x - 1)^2(x + 1)$ or $\mu_P = (x + 1)^2(x - 1)$.
		\item $P \not \in \Omega^3 \cup \Omega^4$ if and only if 
		\begin{enumerate}
			\item $\mu_P = (x^2+1)(x\pm 1)^2$ or
			\item $\mu_P = (x - 1)^2(x + 1)^2$ and $P$ and $-P$ are conjugate.
		\end{enumerate}
	\end{enumerate}
\end{remark}

We obtained \ref{remark} in the nineties using a c-program on a BS2000 mainframe
(runtime 8 h) and recently checked it again using  the computer algebra system GAP \cite{GAP2025} on a meteorlake cpu (runtime less than 30 s). It is not difficult to prove \ref{remark} by hand once you know what to show.

Using GAP, we have also shown the following.

\begin{remark}
	$\SpG(6,3), \SpG(10,3)$, $\SpG(6,5)$ and $\SpG(6,7)$ are 4-reflectional.
\end{remark}	

It follows that $\SpG(2n,3)$ is 4-reflectional for $2n \ge 6$, and that $\SpG(2n,5)$ and $\SpG(2n,7)$ are 4-reflectional for $2n \ne 10$.

\ifdraft{\listoflabels}{}

\end{document}
\typeout{get arXiv to do 4 passes: Label(s) may have changed. Rerun}